\documentclass[11pt,reqno]{amsart}

\usepackage[utf8]{inputenc}
\usepackage{amsthm,amssymb,amsmath,babel,amscd,mathtools}
\usepackage[titletoc,title]{appendix}
\usepackage{dsfont}
\usepackage{comment}
\usepackage{graphicx}
\usepackage[useregional]{datetime2}
\usepackage{enumerate}
\usepackage{todonotes}
\usepackage{csquotes}
\usepackage{microtype}

\newtheorem{theorem}{Theorem}
\newtheorem{lemma}[theorem]{Lemma}
\theoremstyle{definition}

\theoremstyle{remark}
\newtheorem*{remark}{Remark}

\usepackage{xcolor}

\DeclarePairedDelimiter\norm{\lVert}{\rVert}
\newcommand{\msin}{\frac{m}{\sin{\theta}}}
\title{Non-monotonicity for the\\3D magnetic Robin Laplacian}
\author{Germ\'an Miranda}
\begin{document}
	
	\begin{abstract}
		Previous works provided several counterexamples to monoto\-nicity of the lowest eigenvalue for the magnetic Laplacian in the two-dimensional case. However, the three-dimensional case is less studied. We use the results obtained by Helffer, Kachmar and Raymond to provide one of the first counterexamples in 3D. Considering the Robin  magnetic Laplacian on the unit ball with a constant magnetic field, we show the non-monotonicity of the lowest eigenvalue asymptotics when the Robin parameter tends to \(+\infty\). 
	\end{abstract}
	\maketitle
	
	\section{Introduction}
	In \cite{helffer2020magnetic} the authors considered the Robin magnetic Laplacian in the unit ball \(\Omega =\{ x\in \mathbb{R}^3 : |x| < 1\}\)  with a smooth magnetic field. They established a precise asymptotics for the lowest eigenvalue of the operator when the Robin parameter goes to \(+\infty\) (see section \ref{Discussion}). In particular, if the magnetic field is uniform of strength \(b>0\), the asymptotics for the lowest eigenvalue, \(\lambda(\gamma,b)\), is given by
	\[\lambda(\gamma,b) = -\gamma^{2} + 2 \gamma + \mathfrak{e}(b) + o(1) ,\]
	as \(\gamma \rightarrow +\infty\), where \(\mathfrak{e}(b)=\inf_{m\in \mathbb{Z}} \lambda_m(b)\) and \(\lambda_{m}(b)\) is the effective eigenvalue defined as
	\begin{equation}\label{Lambdam}
		\lambda_m (b) = \inf_{f\in \mathcal{D}(q_{m,b}) \setminus \{0\}} \frac{q_{m,b}(f)}{\norm{f}^2_{\mathcal{H}}},
	\end{equation}
	where \(\mathcal{H} = L^2((0,\pi); \sin{\theta}\, d\theta)\),
	\begin{equation}\label{Eq 1}
		q_{m,b}(f) = \int_0^{\pi} \Bigl(|f'(\theta)|^2 + \Bigl(\frac{m}{\sin{\theta}}- \frac{b}{2}\Bigr)^2 |f|^2\Bigr) \sin{\theta}\, d\theta , 
	\end{equation}
	and 
	\begin{equation}\label{domain qmb}
		\mathcal{D}(q_{m,b})= \begin{cases} \{f\in \mathcal{H} : \frac{1}{\sin{\theta}}f, f'\in \mathcal{H}\} & \text{ if }m\neq 0 \\ \{ f\in \mathcal{H}: f'\in \mathcal{H}\} & \text{ if } m=0
		\end{cases}.  
	\end{equation}
	We denote by \(S_{m,b}\) the Friedrichs extension associated with \(q_{m,b}\).
	
	Numerical computations stated in \cite{helffer2020magnetic} suggest a non-monotonic behaviour. Our goal is to give a formal proof of this.
	\begin{theorem}\label{Main theorem}
		The function \(b\mapsto \mathfrak{e}(b)\) is non-monotonic.
	\end{theorem}
	\subsection{Discussion and motivation}\label{Discussion}
	Let us put the result into context. The identification of domains for which the lowest eigenvalue for the magnetic Laplacian is monotone or not with respect to the strength of the magnetic field has been actively studied in the past years. For weak magnetic fields, diamagnetic inequality gives a monotonic behaviour. Moreover, for strong magnetic fields, monotonicity, known as strong diamagnetism, has also been proved for a large variety of domains in \(\mathbb{R}^2\) \cite{Assaad,BonnaillieFournais, FournaisHelffer} and also in \(\mathbb{R}^3\) \cite{FournaisHelffer3D, FournaisPersson}. A detailed discussion and summary of this can be found in~\cite{FournaisHelfferBook}.
	
	For a magnetic field, which strength lies in the middle region, less is known. Non-monotone phase transitions occur in domains having specific topo\-logical properties. A famous example of these phenomena is the Little--Parks effect for 2D annuli \cite{Erdos,FournaisSundqvist,LittleParks}, where an oscillatory behaviour in the critical temperature of the superconductor appears as the magnetic field varies. Similar phenomena are observed for thin domains \cite{HelfferKachmar}. In the disc, counterexamples can be found applying a non-uniform magnetic field \cite{FournaisSundqvist} or by imposing Robin boundary condition with a strong coupling parameter \cite{KachmarSundqvist}. The topological defects can also be induced by an Aharonov--Bohm magnetic potential \cite{KachmarPan2, KachmarPan1}.
	
	However, counterexamples in three dimensions are less studied. Using the asymptotics of the lowest eigenvalue for the magnetic Robin Laplacian when the Robin parameter goes to \(+\infty\), obtained in \cite{helffer2020magnetic}, we are able to provide one of the first counterexamples.
	
	The magnetic Robin Laplacian is given by
	\[P_{\gamma} =(-i\nabla + \mathbf{a})^2 \]
	with domain 
	\[\mathcal{D}(P_{\gamma}) = \{u \in H^2(\Omega) : i\mathbf{n} \cdot (-i\nabla+ \mathbf{a})u + \gamma u = 0 \text{ on } \partial\Omega\},\]
	where \(\gamma > 0\) is the Robin parameter, \(\mathbf{n}\) the unit outward pointing normal vector of \(\partial\Omega\)  and \(\mathbf{a} \) is any magnetic vector potentital that generates the magnetic field \((0,0,b)\). 
	
	The text is organized as follows. In section~\ref{Section 1} we introduce the needed definitions and auxiliary results regarding the spectrum of the self-adjoint operator associated  to \(q_{m,b}\). In section \ref{Section 2} we prove Theorem \ref{Main theorem}. In the appendix, we study the self-adjoint extensions of operators associated to \( \mathfrak{q}^b_h (u)\).
	\section{Definitions and preliminary results}\label{Section 1}
	We want to study the differential expression \(L_{m,b}\) associated with the quadratic form \(q_{m,b}\). 
	
	Integrating by parts, we see that the operator
	\begin{equation}\label{Eq 2}
		L_{m,b} u = r(\theta)^{-1} \Big [ -(p(\theta) u'(\theta))' + q(\theta) u(\theta) \Big ]
	\end{equation}
	with \( r(\theta) = \sin{\theta}\), \(p(\theta) = \sin{\theta} \) and \(q(\theta) = \sin{\theta}\left(m/\sin{\theta}- b/2 \right)^2\), is associated with the quadratic form \(q_{m,b}\). Hence, finding \(\lambda_m (b)\) is equivalent to solve the self--adjoint Sturm--Liouville eigenvalue problem 
	\begin{equation}\label{Eq 3}
		(L_{m,b}-\lambda) u = 0 
	\end{equation}
	in \(\mathcal{H}  = L^2((0,\pi); \sin{\theta}\, d\theta) \), where \(\lambda \in \mathbb{R}\). Because of this, it is useful to introduce the maximal operator \(T_{m,b}\) and the pre-minimal operator \(T'_{0,m,b}\) acting as \( L_{m,b}\), with domains
	\begin{equation}\label{Eq 4} 
		\mathcal{D}(T_{m,b}):= \{f\in \mathcal{H}  : f, f' \in AC_{loc}(0,\pi), L_{m,b} f \in \mathcal{H} \} ,
	\end{equation}
	\begin{equation}\label{Eq 5}  \mathcal{D}(T'_{0, m,b}):=  \{f\in \mathcal{D}(T_{m,b}) : f \text{ has compact support}\},
	\end{equation}
	where \(AC_{loc}(0,\pi)\) denotes the space of locally absolutely continuous functions on \((0, \pi)\) (see \cite[Chapter 8]{Brezis} for more details). 
	
	The pre-minimal operator is symmetric, densely defined and \((T'_{0, m,b})^{\ast} = T_{m,b}\). Thus, any self--adjoint extension \(A\) of \(T'_{0, m,b}\) satisfies
	\[ T'_{0, m,b}\subset A = A^{\ast} \subset T_{m,b} .\]
	In particular, \(T'_{0, m,b}\subset S_{m,b} \subset T_{m,b}\), where \(S_{m,b}\) corresponds to the Friedrichs extension associated with \(q_{m,b}\) with the form domain introduced in \eqref{domain qmb}.
	
	We also define the minimal operator \( T_{0,m,b} \) as the closure of the pre-minimal operator \(T'_{0, m,b}\) for \(m\neq 0\).
	
	\begin{lemma}\label{Lemma2}
		If \(m\in \mathbb{Z}\) and \(b\in\mathbb{R}\), then \(S_{m,b}\) has pure discrete spectrum.
	\end{lemma}
	\begin{proof}
		Consider the operator \(S_{m,b}\) restricted to \((0, \tfrac{1}{2}\pi)\), and let  \(T_{(0,\frac{\pi}{2}),m,b}\)  and \(T'_{(0,\frac{\pi}{2}),m,b}\)  be the maximal and pre-minimal operator corresponding to \(L_{m,b}\) restricted to this interval. \( T'_{(0,\frac{\pi}{2}),m,b}\) is a symmetric, densely defined and semi-bounded operator on \(L^2 ((0, \frac{\pi}{2}); \sin{\theta}\, d\theta)\). Thus, we can obtain its Friederichs extension \(S_{(0 , \frac{\pi}{2}),m,b}\). Since
		\begin{equation}\label{Eq 7}
			\int_0^{\frac{\pi}{2}} r(x)\left( \int_{x}^{\frac{\pi}{2}} \frac{1}{p(t)} \, dt\right)\,dx = \int_0^{\frac{\pi}{2}} \sin{x}\left( \int_{x}^{\frac{\pi}{2}} \frac{1}{\sin{t}}\, dt\right)\,dx = \ln{2},
		\end{equation}
		adapting \cite[Theorem 1]{baxley} to the appropriated domain, we have that $0$ belongs to the resolvent set of $S_{(0,\pi/2),m,b}$ and that its inverse is compact. Thus $S_{(0,\pi/2),m,b}$ has empty essential spectrum. By symmetry, the operator $S_{(\pi/2,\pi),m,b}$ also has empty essential spectrum. This implies that $S_{m,b} $ has empty essential spectrum \cite[Theorem 9.11]{Teschl}.
	\end{proof}
\section{Nonmonotonicity. Proof of Theorem 1.}\label{Section 2}
In Figure~\ref{fig:my_label}, we see that the crossing between \(\lambda_0(b)\) and \(\lambda_1(b)\) occurs on the interval \([0,2]\), and we expect that \(\mathfrak{e}(b)\) equals \(\lambda_0(b)\) or \(\lambda_1 (b)\) on this interval. Because of this, we study the case \(0\leq b <2\) in order to check the non-monotonicity.

Let \(V_{m,b}(\theta):= \left(\frac{m}{\sin{\theta}}- \frac{b}{2}\right)^2\). If \(m=0\) then \(V_{0,b}(\theta) = \frac{b^2}{4}\), and this implies that \(\lambda_0(b) = \frac{b^2}{4}\) (the constant function is an eigenfunction) where \(b\geq 0\). 

It is sufficient to study the case $m\geq 0$ since for \(m \leq -1\) we have that \(\lambda_m(b) > \frac{b^2}{4}\) since \(V_{m,b}(\theta)>b^2/4\), so
\[ \mathfrak{e}(b) = \inf_{m\in \mathbb{N}\cup \{0\}}\lambda_m(b).\]

\begin{figure}[h]
	\centering
	\includegraphics[width=0.5\textwidth]{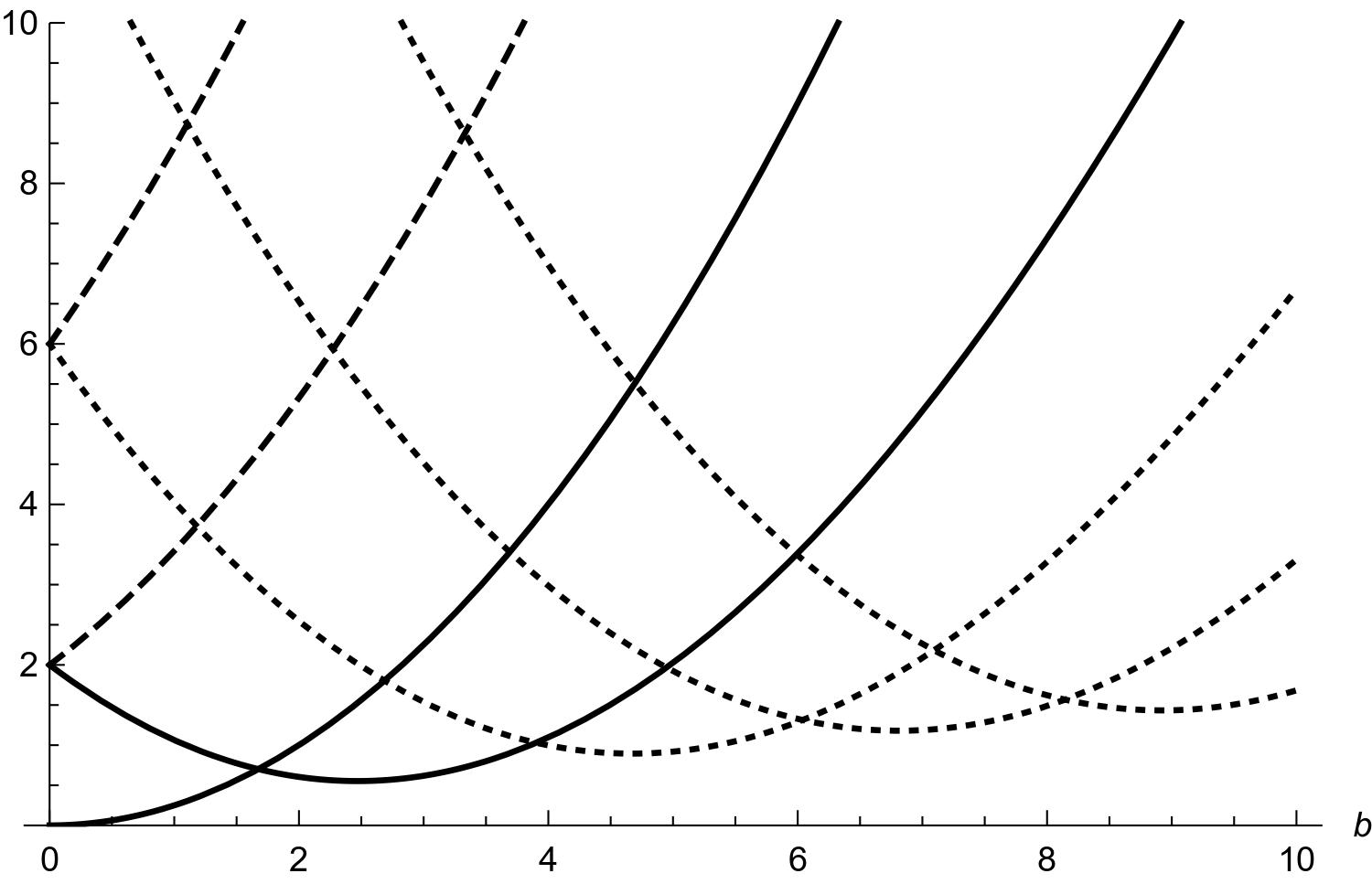}
	\caption{Numerical computations of \(\lambda_m(b)\) for \(-2\leq m \leq 4\). \(\lambda_0(b)\) and \(\lambda_1(b)\) are depicted as solid lines. Dotted lines represent values of \(\lambda_m(b)\) for \(m\geq 2\), and the dashed ones of \(m=-1, -2\).  }
	\label{fig:my_label}
\end{figure}
\begin{enumerate}[(i)]
	\item \( 0\leq b< 1 \): 
	
	We note that for \(m\geq 1\)
	\[\left(\frac{m}{\sin{\theta}}- \frac{b}{2}\right)^2 = \frac{m}{\sin{\theta}}\left(\frac{m}{\sin{\theta}} - b\right) + \frac{b^2}{4} > \frac{b^2}{4}.\]
	Hence, $\lambda_m(b)>\lambda_0(b)$ and \( \mathfrak{e}(b) = \lambda_0 (b) = \frac{b^2}{4}.\) 
	Observe that \(b\mapsto\mathfrak{e}(b)\) is increasing in this interval.
	\item\( 1\leq b< 2 \):
	
	Let \(m \geq 2\), then \(\frac{m}{\sin{\theta}} \geq 2\) and \(b<2\), so
	\[ \left( \msin - \frac{b}{2}\right)^2 > 1 >\frac{b^2}{4},  \]
	so \(\mathfrak{e}(b)\) equals \(\lambda_0(b)\) or \( \lambda_1(b)\) in this interval.
	
	Using \(g(\theta) = \sin \theta \) as a trial function, one can find a $b_0 \in [1,2]$  such that $\lambda_1(b) < \frac{b^2}{4}= \lambda_0 (b)$ for \(b\in (b_0,2)\). Indeed
	\[\frac{q_{1,b}(g)}{\norm{g}^2_{ \mathcal{H}} } =  \frac{b^2}{4} - \frac{3}{8}\pi b + 2,\]
	which is smaller than \(\lambda_0(b)\) for \(b> \tfrac{16}{3}\pi\simeq 1.7\). Thus, we have a crossing point.
	
	Now, we want to increase $b_0$ by a small positive amount \(\delta >0\) to see \(\lambda_1 (b_0) > \lambda_1 (b_0 + \delta )  \). This would mean that we do not have monotonicity. 
	We can rearrange \( V_{1,b_0 + \delta}(\theta)\) as
	\[ V_{1,b_0 + \delta}(\theta) = \left(\frac{1}{\sin{\theta}}- \frac{b_0}{2}\right)^2 + \frac{\delta^2}{4} - \delta \left(\frac{1}{\sin{\theta}} - \frac{b_0}{2}\right). \] 
	Having this in mind, the operator associated with \(q_{1,b_0 +\delta}\) is given by
	\begin{equation}\label{Eq 9}
		L_{1, b_0+\delta} = -\frac{d^2}{d\theta^2} - \frac{\cos{\theta}}{\sin{\theta}} \frac{d}{d\theta} + V_{1,b_0+\delta}(\theta) ,
	\end{equation}
	
	Let \(f_0 \in \mathcal{D} (L_{1, b_0})\) be a normalized positive eigenfunction of \( L_{1, b_0}\) with eigenvalue \(\lambda_1(b_0)\). 
	Then 
	\[ \begin{aligned}
		q_{1, b_0 + \delta } (f_0) &= \int_0^{\pi} \left(|f_0'(\theta)|^2 + V_{1,b_0 + \delta}(\theta)  |f_0|^2\right) \sin{\theta} \, d\theta\\
		& =  q_{1, b_0  } (f_0) + \frac{\delta^2}{4} - \delta  \int_0^{\pi} \underbrace{\left(\frac{1}{\sin{\theta}}- \frac{b_0}{2}\right)}_{>0} |f_0(\theta)|^2\sin{\theta}\, d\theta\\
		& = \lambda_1(b_0) - \delta C_{b_0} + \frac{\delta^2}{4},
	\end{aligned}
	\]
	where \(C_{b_0}>0\). Thus, for $\delta>0$ small enough we get \(\mathfrak{e}(b_0+\delta)\leq \lambda_1 (b_0 +\delta) < \lambda_1(b_0)=\mathfrak{e}(b_0)\). We conclude that Theorem~\ref{Main theorem} holds. 
\end{enumerate}
\begin{remark}
	This argument can be extended to other crossing points \(b_m\) between \(\lambda_m(b)\) and \(\lambda_{m+1}(b)\), where \(m=1,2,\ldots\), as long as \(m+1 >\frac{1}{2}b_m\). The numerical computations depicted in the Figure 1, suggest that such crossing points are expected. This would mean that the \(\mathfrak{e}(b)\) is not only non-monotonic on \((0,2)\), but also on intervals with endpoints greater than \(2\). 
	
	The main difficulty to extend the result for higher \(b\) is to find lower bounds of \(\lambda_m(b)\) for \(m\geq 1\). Our proof is based on the fact that for \(m =0\) we know explicitly the lowest eigenvalue. 
\end{remark}
\appendix
\section{Extensions of \(T'_{0,m,b}\)}
Although it is not strictly necessary for the proof of Theorem 1, it is interesting to understand better \(S_{m,b}\). In particular, one can show that the pre-minimal operator \(T'_{0,m,b}\) is essentially self-adjoint for all \(m\neq 0\). 

We remind that the endpoint \(0\) (alternatively \(\pi\)) is in the limit circle case (l.c.c.) if all solutions of \(L_{m,b} u = \lambda u\), with \(\lambda\in \mathbb{C}\), are in \(L^2((0,d); \sin{\theta}\,d\theta)\) for some (and hence any) \(d\in (0,\pi)\). An endpoint is in the limit point case (l.p.c.) if it is not in the limit circle case~\cite{Zettl}.
\begin{lemma}\label{Lemma1}
	For \(m\neq 0 \), \(T'_{0,m,b}\) is essentially self-adjoint. If \(m=0\) then both endpoints are in the limit circle case. 
\end{lemma}
\begin{proof}
	In order to classify our Sturm--Liouville operator, it is useful to apply the Liouville transformation. This is possible since 
	\[p(\theta), p'(\theta), r(\theta), r'(\theta) \in AC_{loc}(0,\pi)  \text{ and } p(\theta), r(\theta) >0\] for all \(\theta \in (0, \pi)\) (see \cite[section 7]{Everitt}). Applying the unitary transformation \(u(\theta) = w(\theta) \sin^{-1/2}{\theta} \),  we can rewrite equation \eqref{Eq 3} in the Liouville normal form 
	\begin{equation}\label{Eq 6}
		w''(\theta) + [\lambda - \hat{q}(\theta)] w(\theta) =0, 
	\end{equation}
	where
	\begin{align*} \hat{q}(\theta) &= \frac{q(\theta)}{r(\theta)} - \frac{(p(\theta)r(\theta))^{1/4} }{r(\theta)}  \left(p(\theta) \left((p(\theta)r(\theta))^{-1/4}\right)'\right)' \\
		&= \left(\msin - \frac{b}{2}\right)^2 - \frac{1}{4}\frac{\cos^2\theta}{\sin^2 {\theta }} - \frac{1}{2}. 
	\end{align*}
	Due to the symmetry of the problem, we can focus on the interval \((0, \frac{\pi}{2})\). All the results stated at \(0\) can be analogously stated at \(\pi\). We distinguish three cases:
	
	\begin{enumerate}[(i)]
		\item  \( m \neq 1\). Observe that for \(\theta\) close to \(0\)
		\[ \hat{q}(\theta) \sim\left( m^2 - \frac{1}{4}\right) \frac{1}{\theta^2} + \mathcal{O}\left(\frac{1}{\theta}\right) .\]
		Thus, we can find a constant \(C\) such that \(\hat{q}(\theta)\geq C +\frac{3}{4}\frac{1}{\theta^2}\) for \(\theta\) close to~\(0\). We satisfy the conditions of \cite[Theorem 3.3]{WeidmannApproxRegular}, which ensures that \(L_{m,b}\) is in the limit point case. 
		
		\item \(m=1\). A constant \(C\) as in the previous case cannot be found since
		\[ \hat{q}(\theta) =  \frac{3}{4\sin^2{\theta}} - \frac{b}{\sin{\theta}} + \frac{b^2 - 1}{4} .\] 
		However, using asymptotics methods for \(\theta\) close to \(0\), we see that the general solution of the ODE generated by the Sturm-Liouville eigenvalue problem considering the Liouville form of \(L_{1,b}\) is given by 
		\[u(\theta) = c_1(\theta^{3/2} - \frac{b}{3} \theta^{5/2} + \mathcal{O}(\theta^{7/2})) + c_2(\theta^{-1/2} - b \theta^{1/2} + \mathcal{O}(\theta^{3/2})) ,\]
		where \(c_1\) and \(c_2\) are suitable constants. Note that one of the two independent solutions is not in \(L^2\). Thus, \(L_{1,b}\) is in the limit point case. 

	We fulfill the conditions of \cite[Theorem 10.4.1 (i)]{Zettl}, which states that if \(L_{m,b}\) is in l.p.c. at \(0\) and \(\pi\), this is equivalent to have deficiency indices \( (0,0)\). Hence, \(T'_{0,m,b}\) is essentially self-adjoint and the minimal operator \(T_{0,m,b}\) is the only self-adjoint extension.
	\item \(m=0\). In this case
	\[\hat{q}(\theta)=  -\frac{1}{4\sin^2{\theta}}  + \frac{b^2 - 1}{4} . \]
	Using again asymptotics methods for \(\theta\) close to \(0\), we see that the general solution considering the Liouville form of \(L_{0,b}\) is given by 
	\[u(\theta) = d_1(\theta^{1/2} - \mathcal{O}( \theta^{5/2})) + d_2 \log{\theta} \left(\theta^{1/2} + \mathcal{O}(\theta^{5/2})\right),\]
	where \(d_1\) and \(d_2\) are again suitable constants. Observe that in this case both solutions are in \(L^2\). Thus, \(L_{0,b}\) is in the limit circle case. 
	\end{enumerate}	
\end{proof}
\begin{remark}
	Note that the Friedrichs extension \(S_{m,b}\) introduced before must coincide with  \(T_{0,m,b}\). For \(m=0\), `self- adjoint boundary conditions' (see \cite[Section 10.4]{Zettl} for more details) are used to describe the domain of \(S_{0,b}\). 
\end{remark}
\subsection*{Acknowledgment}
I would like to thank A. Kachmar for the suggestion and discussion of the problem.

\bibliographystyle{plain} 
\bibliography{main2} 

\end{document}